\documentclass[12pt,reqno]{amsart}
\usepackage{amssymb,latexsym,amsmath,amscd, graphicx}
\usepackage[colorlinks=true,pagebackref,hyperindex]{hyperref}
\usepackage{verbatim}
\usepackage{xypic}
\usepackage{tikz}

\usepackage{amsfonts}
\usepackage{amscd}

\renewcommand{\phi}{\varphi}
\renewcommand{\epsilon}{\varepsilon}
\renewcommand{\theta}{\vartheta}

\def\ZZ{{\mathbf Z}}

\def\RR{{\mathbf R}}

\def\PP{{\mathbf P}}

\def\cO{\mathcal{O}}

\def\frm{\mathfrak{m}}

\newtheorem{lemma}{Lemma}[section]
\newtheorem{theorem}[lemma]{Theorem}
\newtheorem{corollary}[lemma]{Corollary}
\newtheorem{proposition}[lemma]{Proposition}
\newtheorem*{claim}{Claim}

\theoremstyle{definition}
\newtheorem{definition}[lemma]{Definition}
\newtheorem{remark}[lemma]{Remark}

\newtheorem{example}[lemma]{Example}

\theoremstyle{remark}
\newtheorem*{remark*}{Remark}
\newtheorem*{note*}{Note}

\numberwithin{equation}{section}

\newcommand{\cf}{{\itshape cf.} }

\begin{document}

\title{A Frobenius variant of Seshadri constants}

  \author[M.~Musta\c{t}\u{a}]{Mircea~Musta\c{t}\u{a}}
\address{Department of Mathematics, University of Michigan,
Ann Arbor, MI 48109, USA}

\email{mmustata@umich.edu}

\author[K.~Schwede]{Karl~Schwede}
\address{Department of Mathematics, The Pennsylvania State University,
University Park, PA 16802, USA}

\email{schwede@math.psu.edu}

\begin{abstract}
We define and study a version of Seshadri constant for ample line bundles in positive characteristic. We prove that lower bounds for this constant imply the global generation or very ampleness of the corresponding adjoint line bundle. As a consequence, we deduce that the
criterion for global generation and very ampleness of adjoint line bundles in terms of usual
Seshadri constants holds also in positive characteristic.
\end{abstract}
\thanks{The first author was partially supported by NSF research grant
no: 1068190 and by a Packard Fellowship.}
\thanks{The second author was partially support by NSF research grant no: 1064485.}


\maketitle

\markboth{M.~MUSTA\c{T}\u{A} AND K.~SCHWEDE}{A FROBENIUS VARIANT OF SESHADRI CONSTANTS}

\section{Introduction}

Let $L$ be an ample line bundle on an $n$-dimensional projective variety $X$ (defined over an algebraically closed field $k$ of positive characteristic). The Seshadri constant
$\epsilon(L; x)$ of $L$
at a smooth point
$x\in X$ measures the local positivity of $L$ at $x$. Introduced by Demailly twenty years ago in
\cite{Demailly}, it has generated a lot of interest: see, for example \cite{positivity} and \cite{B+}
for some of the results and open problems involving this invariant.

If $\pi\colon {\rm Bl}_x(X)\to X$ is the blow-up of $X$ at $x$, with exceptional divisor $E$, then
the Seshadri constant $\epsilon(L; x)$ is defined as
$$\epsilon(L,x)=\sup\{t>0\mid \pi^*(L)(-tE)\,\text{is nef}\}.$$
We will be especially interested in an equivalent description in terms of separation of jets.
Recall that one says that a power $L^m$ separates $\ell$-jets at $x$ if the restriction map
$$H^0(X, L^m)\to H^0(X,L^m\otimes\cO_X/\frm_x^{\ell+1})$$
is surjective, where $\frm_x$ is the ideal defining $x$. If $s(L^m; x)$ is the largest $\ell$
such that $L^m$ separates $\ell$-jets at $x$, then
$$\epsilon(L; x)=\lim_{m\to\infty}\frac{s(L^m; x)}{m}=\sup_{m\geq 1}\frac{s(L^m; x)}{m}.$$

Part of the interest in Seshadri constants comes from its connection to statements about the base-point freeness or very ampleness of adjoint line bundles. If the ground field has characteristic zero and $X$ is smooth, then it is an easy consequence of the Kawamata-Viehweg vanishing theorem that
if $\epsilon(L; x)>n$, then $\omega_X\otimes L$ is globally generated at $x$,
and if $\epsilon(L; x)>2n$ for every $x\in X$, then $\omega_X\otimes L$
is very ample. One of the main results in our paper says that the same implications
hold also in positive characteristic.

We prove this by studying another version of Seshadri constant in positive characteristic,
which is designed to take advantage of the Frobenius morphism. Suppose that the ground field has characteristic $p>0$. For a positive integer $e$ and a smooth point $x\in X$, we say that a power
$L^m$ separates $p^e$-Frobenius jets at $x$ if the restriction map
$$H^0(X,L^m)\to H^0(X,L^m\otimes\cO_X/\frm_x^{[p^e]})$$
is surjective, where $\frm_x^{[p^e]}$ is the ideal locally generated by the $p^e$-powers of the generators of $\frm_x$. We denote by $s_F(L^m; x)$ the largest $e$ such that
$L^m$ separates $p^e$-Frobenius jets at $x$. With this notation, the Frobenius-Seshadri
constant is defined as
$$\epsilon_F(L; x):=\sup_{m\geq 1}\frac{p^{s_F(L^m; x)}-1}{m}=\limsup_{m\to\infty}
\frac{p^{s_F(L^m; x)}-1}{m}.$$
The following inequalities between the two versions
of Seshadri constants follow easily from definition
$$\frac{\epsilon(L; x)}{n}\leq \epsilon_F(L; x)\leq\epsilon(L; x),$$
and we give examples when either of the two bounds is achieved.
One can further compare the two invariants in the case of torus-invariant points on
smooth toric varieties, when both of them can be described in terms of the polytope $P$ attached to the line bundle (and suitably normalized). In this case the usual Seshadri constant is obtained by comparing $P$ with the simplex $\{u=(u_1,\ldots,u_n)\in\RR_{\geq 0}^n\mid u_1+\ldots+u_n\leq 1\}$, while the Frobenius
Seshadri constant is obtained by comparing $P$ with the cube $[0,1]^n$
(see Theorem~\ref{toric} below for the precise statement).

The Frobenius-Seshadri constant satisfies some of the basic properties of the usual Seshadri constant, though in some cases the proofs are a bit more subtle. For example, we show that
$\epsilon_F(L; x)$ only depends on the numerical equivalence class of $L$
and $\epsilon_F(L^m;x)=m\cdot\epsilon_F(L,x)$ for every $m\geq 1$. The following theorem is our main result (see Theorem~\ref{adjoint} below for these statements and other
related ones).

\noindent{\bf Theorem}. Let $L$ be an ample line bundle on the smooth projective
variety $X$, defined over a field of positive characteristic.
\begin{enumerate}
\item[i)] If $\epsilon_F(L; x)>1$, then $\omega_X\otimes L$ is globally generated at $x$.
\item[ii)] If $\epsilon_F(L; x)>2$ for every $x\in X$, then $\omega_X\otimes L$ is very ample.
\end{enumerate}
Note that since $\epsilon_F(L; x)\geq\frac{\epsilon(L; x)}{n}$, the above theorem implies the
positive characteristic version of the facts we mentioned above. Namely, if $\epsilon(L; x)>n$, then $\omega_X\otimes L$
is globally generated at $x$, and if $\epsilon(L; x)>2n$ at every $x\in X$, then
$\omega_X\otimes L$ is very ample.

In light of the above theorem, it would be interesting to obtain
lower bounds for either of the two versions of Seshadri constants at very general points
of $X$. Recall that over an uncountable field of characteristic zero, it is expected that
$\epsilon(L; x)\geq 1$ if $x\in X$ is a very general point. This is known if $X$ is a surface
\cite{EL}. When $n\geq 3$, it is shown in \cite{EKL} that $\epsilon(L; x)\geq\frac{1}{n}$
when $x$ is very general in $X$. However, the proofs of both these results make essential use
of the characteristic zero assumption.

The paper is structured as follows. In the next section we introduce the Frobenius-Seshadri
constant and prove some basic properties. In fact, we define a slightly more general version,
in which we consider a finite set of smooth points. In the third section we prove our main result
relating lower bounds on Seshadri constants to global generation and very ampleness properties
of the corresponding adjoint line bundles. In the last section we describe the Frobenius-Seshadri
constant at a torus-fixed point on a smooth toric variety.

\subsection*{Acknowledgment}
This project originated in discussions held during the AIM workshop ``Relating test ideals and multiplier ideals". We are indebted to AIM for organizing this event.
We would also like to thank Bhargav Bhatt and Rob Lazarsfeld for discussions related to this work.

\section{Definition and basic properties}

We begin by recalling the definition and some basic properties of the Seshadri constant
of an ample line bundle.
For details and a nice introduction to this topic, we refer to
\cite[\S 5]{positivity}.

Let $X$ be an $n$-dimensional projective variety (assumed irreducible and reduced)
over an algebraically closed field $k$.
Consider an ample line bundle $L$ on $X$ and $x\in X$ a (closed) point.
If $\pi\colon Y\to X$ is the blow-up of $X$ at $x$, with exceptional divisor $E$,
then the \emph{Seshadri constant} of $L$ at $x$ is
\begin{equation}\label{def_Seshadri}
\epsilon(L; x):=\sup\{\alpha\in\RR_{\geq 0}\mid \pi^*(L)(-\alpha E)\,\text{is nef}\}.
\end{equation}
It is easy to see that since $L$ is ample, we have
$\epsilon(L; x)>0$. Furthermore, one has
$$\epsilon(L; x)=\inf_{Y\ni x}\frac{(L\vert_Y^{\dim(Y)})}{{\rm mult}_x(Y)},$$
where the infimum is over all positive-dimensional irreducible
closed subsets $Y$ of $X$ containing $x$.

We briefly recall some basic properties of Seshadri constants for the purpose of comparison and ease of reference.  All this material can be found in \cite[\S 5]{positivity}.
\begin{proposition}
\label{prop.BasicPropertiesOfSeshadri}
Suppose that  $L$ is an ample line bundle on the projective variety $X$, and let $x\in X$.
\begin{enumerate}
\item $\epsilon(L^m; x) = m \cdot \epsilon(L; x)$.  \cite[Example 5.1.4]{positivity}, \cf Proposition \ref{prop2}.
\item $\epsilon(L; x)$ depends only on the numerical equivalence class of $L$.  \cite[Example 5.1.3]{positivity}, \cf Proposition \ref{numerical}.
\end{enumerate}
\end{proposition}

\begin{proposition}\label{prop.BasicPropertiesOfSeshadri_part2}
Suppose, in addition, that $X$ is smooth and the ground field $k$ has characteristic zero.
\begin{enumerate}
\item If $\epsilon(L; x)>\dim(X)$, then the line bundle $\omega_X\otimes L$ is globally generated at $x$. \cite[Proposition 5.1.19(i)]{positivity}, \cf Corollary \ref{cor.adjoint}${\rm (i)}$.
\item If $\epsilon(L; x) > 2 \dim X$ then the rational map defined by $\omega_X\otimes L$ is birational onto its image. \cite[Proposition 5.1.19(ii)]{positivity},
 \cf Corollary \ref{cor.adjoint}${\rm (ii)}$.
\item If $\epsilon(L; x) > 2 \dim X$ for every point $x \in X$, then $\omega_X\otimes L$ is very ample.  \cite[Proposition 5.1.19(iii)]{positivity}, \cf Corollary \ref{cor.adjoint}${\rm (iii)}$.
\end{enumerate}
\end{proposition}

The Seshadri constant can be alternatively described in terms of jet separation, as follows. One says that a line bundle $A$ on $X$ separates $\ell$ jets at $x\in X$ if the restriction map
\begin{equation}\label{separation_of_jets}
H^0(X, A)\to H^0(X,A\otimes\cO_X/\frm_x^{\ell+1})
\end{equation}
is surjective, where $\frm_x$ is the ideal defining $x$. Let $s(A; x)$ be the largest $\ell\geq 0$
such that $A$ separates $\ell$ jets at $x$ (with the convention $s(A; x)=0$ if there is no such
$\ell$). It is proved in
\cite[Theorem~5.1.17]{positivity} \cf \cite{Demailly}, that if $x\in X$ is a smooth point, then
\begin{equation}\label{formula_jet_separation}
\epsilon(L; x)=\lim_{m\to\infty}\frac{s(L^m; x)}{m}=\sup_{m\geq 1}\frac{s(L^m; x)}{m}
\end{equation}
(note that the proof therein is characteristic-free).

\bigskip

We turn to the definition of the Frobenius version of the Seshadri constant.
From now on we assume that the ground field has characteristic $p>0$.
We denote by $F\colon X\to X$ the absolute Frobenius morphism, given by the identity on points, and which maps a section
$f$ of $\cO_X$ to $f^p$. This is a finite morphism since $k$ is perfect, and it is flat
on the smooth locus of
$X$ by \cite{KunzCharacterizationsOfRegularLocalRings}. If $J$ is an ideal on $X$ (always assumed to be coherent), we denote by $J^{[p^e]}$ the inverse image
of $J$ by $F^{e}$: if $J$ is locally generated by $(h_i)_{i\in I}$, then $J^{[p^e]}$
is locally generated by $(h_i^{p^e})_{i\in I}$.

The definition of the Frobenius-Seshadri constant is
 modeled
on the above interpretation of the usual Seshadri constant in terms of separation of jets. The main difference is that we replace
the usual powers of the ideal $\frm_x$ defining $x$ by the Frobenius powers.
While we are mostly interested in the Frobenius-Seshadri constant at one point, we will
need the notion in the case of several points in the next section, hence we give the definition
in this slightly more general setting.

Suppose that $Z$ is a finite set of \emph{smooth} closed points on $X$ (with the reduced scheme structure),
defined by the ideal $I_Z$. Given a positive integer $e$, we say that a line bundle $A$ on $X$ separates
$p^e$-Frobenius jets at $Z$ if the restriction map
\begin{equation}\label{restriction_Frobenius}
H^0(X,A)\to H^0(X,A\otimes\cO_X/I_Z^{[p^e]})
\end{equation}
is surjective.

\begin{remark}
If $A$ separates $p^e$-Frobenius jets at $Z$ and $B$ is globally generated, then
$A\otimes B$ also separates $p^e$-Frobenius jets at $Z$. Indeed, there is a section
$t\in H^0(X,B)$ that does not vanish at any point in $Z$, and we have a commutative diagram
\[\label{diag1}
\begin{CD}
H^0(X,A) @>{-\otimes t}>> H^0(X,A\otimes B) \\
@VVV@VVV \\
H^0(X,A\otimes\cO_X/I_Z^{[p^e]})@>>>  H^0(X,A\otimes B\otimes\cO_X/I_Z^{[p^e]})
\end{CD}
\]
in which the bottom horizontal map is an isomorphism by the assumption on $t$, and the left vertical map is surjective by the hypothesis on $A$. Therefore the right vertical map
is surjective.
\end{remark}

Let $s_F(L^m; Z)$ be the largest $e\geq 1$ such that $L^m$ separates
$p^e$-Frobenius jets at $Z$ (if there is no such $e$, then we put
$s_F(L^m ;Z)=0$).  Now we come to the main definition of this paper.

\begin{definition}[Frobenius-Seshadri Constant]
The \emph{Frobenius-Seshadri} constant of $L$ at $Z$ is
$$\epsilon_F(L; Z):=\sup_{m\geq 1} \frac{p^{s_F(L^m; Z)}-1}{m}.$$
\end{definition}

The following lemma gives the analogue in our setting for the inequality
$s(L^{mr}; x)\geq r\cdot s(L^m; x)$ for all positive integers $m$ and $r$.

\begin{lemma}\label{lem1}
If $e=s_F(L^m; Z)$ and $d_r=\frac{p^{re}-1}{p^e-1}$ for some positive integer $r$, then
$$s_F\left(L^{md_r}; Z \right)\geq re.$$
\end{lemma}

\begin{proof}
If $e=0$, then the assertion is clear, hence we may assume that $e\geq 1$.
It is enough to show that for every $x\in Z$, if $Z_x=Z\smallsetminus \{x\}$ and
$\Gamma^r_x:=H^0(X,I_{Z_x}^{[p^{re}]}\otimes L^{md_r})$,
then the map induced by restriction
\begin{equation}\label{eq1_lem1}
\phi_r\colon \Gamma^r_x\to H^0(X,L^{md_r}\otimes\cO_X/\frm_x^{[p^{re}]})
\end{equation}
is surjective, where $\frm_x$ is the ideal defining $x$.
We prove this by induction of $r$. Note that the case $r=1$ follows from hypothesis.
Suppose now that $r\geq 2$.

By assumption, $x$ is a smooth point of $X$.
Let us choose local algebraic coordinates $y_1,\ldots,y_n$ at $x$ (that is,
a regular system of parameters of $\cO_{X,x}$). After choosing an isomorphism
$L_x\simeq\cO_{X,x}$, we identify
 $L^{md_r}\otimes
\cO_X/\frm_x^{[p^{re}]}$ with
$\cO_X/\frm_x^{[p^{re}]}$, which
has a $k$-basis given by $y_1^{a_1}\cdots y_n^{a_n}$, with $0\leq a_i\leq p^{re}-1$
for all $i$. Let us write
$$a_i=a_{i,0}+a_{i,1}p^e+\ldots+a_{i,r-1}p^{e(r-1)}=a_{i,0}+p^ea'_i,$$ with
$0 \leq a_{i,j}\leq p^e-1$ for all $i$ and $j$.  Observe that $0 \leq a_i' \leq p^{(r-1)e} - 1$ for all $i$.

We will prove that $y_1^{a_1}\cdots y_n^{a_n}$ lies in the image of $\phi_r$ by descending induction on $S:=\sum_{i=1}^na_i'$.  This will complete the proof.  The first non-trivial step in the induction starts with $S=n(p^{(r-1)e}-1)$, in which case all $a_i' = p^{(r-1)e}- 1$.

\begin{claim}
Each element in the product $(y_1^{p^e},\ldots,y_n^{p^e})\cdot\prod_{i=1}^ny_i^{p^ea'_i}$is congruent to an element in ${\rm Im}(\phi_r)$ mod $\frm_x^{[p^{re}]}$.
\end{claim}
\begin{proof}[Proof of claim]
To see this, consider
the term
\begin{equation}\label{eq_claim}
y_{\ell}^{p^e} \cdot \prod_{i=1}^ny_i^{p^ea'_i}.
\end{equation}
If $a'_{\ell}=p^{(r-1)e}-1$ (for example, this is the case when
$S=n(p^{(r-1)e}-1)$), then $p^e+p^ea'_{\ell}=p^{re}$ and the monomial in (\ref{eq_claim})
lies in $\frm_x^{[p^{re}]}$.
On the other hand, if $a'_{\ell}<p^{(r-1)e}-1$, then the monomial in
(\ref{eq_claim}) can be written as
\begin{equation}\label{eq_claim2}
\prod_{i=1}^ny_i^{p^eb'_i}
\end{equation}
with $0 \leq b_i' \leq p^{e(r-1)} - 1$ for every $i$.  We then write $b'_i=\sum_{j=1}^{r-1}b_{i,j}p^{(j-1)e}$ for some $b_{i,j}$
with $0\leq b_{i,j}\leq p^e-1$. Since
$\sum_{i=1}^nb'_i=S+1$, it follows from the inductive hypothesis with respect to $S$
that the monomial in (\ref{eq_claim2}) is congruent to an element in ${\rm Im}(\phi_r)$
mod $\frm_x^{[p^{re}]}$.
This completes the proof of the claim.
\end{proof}
We now return to the proof of the lemma.
Surjectivity of $\phi_{1}$ implies that there is $t_{1}\in \Gamma_x^{1}$ whose image $t_{1,x} \in \cO_{X,x} \simeq L_x$ satisfies
$\prod_{i=1}^ny_i^{a_{i,0}}- t_{1,x} \in (y_1^{p^e},\ldots,y_n^{p^e}) = \mathfrak{m}_x^{[p^e]}$.
Furthermore,
by the inductive assumption with respect to $r$, we can find
$t_2\in\Gamma_x^{r-1}$ such that $\prod_{i=1}^ny_i^{a'_i}$ is congruent to
$\phi_{r-1}(t_2)$ modulo $\frm_x^{[p^{e(r-1)}]}$.  In other words
\[
\prod_{i=1}^ny_i^{a'_i} - t_{2,x} \in \mathfrak{m}_x^{[p^{(r-1)e}]} \subseteq \cO_{X,x}.
\]
In this case $(F^e)^*(t_2) = t_2^{p^{e}}$ lies in
$H^0(I_{Z_x}^{[p^{(r-1)e + e}]}\otimes L^{mp^ed_{r-1}})$.
Observe that $t_1t_2^{p^e}\in \Gamma_x^r \subseteq H^0(X, L^{mdr})$ and also that $\phi_r(t_1t_2^{p^e})$ is simply the residue of $t_{1,x} \cdot t_{2,x}^{p^e} = (t_1 t_{2}^{p^e})_x$ modulo $\mathfrak{m}^{[p^{re}]}$.  We certainly have that
\[
\begin{array}{rcl}
&   & \prod_{i=1}^ny_i^{a_i}-t_{1,x} \cdot t_{2,x}^{p^e} \\
& = & \prod_{i=1}^ny_i^{a_i} -  t_{1,x} \cdot t_{2,x}^{p^e}  + t_{1,x} \cdot \prod_{i=1}^ny_i^{p^e a'_i} - t_{1,x} \cdot \prod_{i=1}^ny_i^{p^e a'_i}\\
& = & \left(\prod_{i=1}^ny_i^{a_{i,0}}- t_{1,x} \right)\cdot  \prod_{i=1}^ny_i^{p^e a'_i} + t_{1,x}\cdot \left( \prod_{i=1}^ny_i^{p^e a'_i} - t_{2,x}^{p^e} \right) \\
& \in & (y_1^{p^e},\ldots,y_n^{p^e})\cdot
\prod_{i=1}^ny_i^{p^ea'_i}+\frm_x^{[p^{re}]}.
\end{array}
\]
The claim then implies that
$\prod_{i=1}^ny_i^{a_i}$ is congruent to an element in
${\rm Im}(\phi_r)$ mod $\frm_x^{[p^{re}]}$. This proves the lemma.
\end{proof}

\begin{proposition}\label{prop1}
Let $L$ be an ample line bundle on the projective variety $X$, and $Z$ a reduced finite set of
smooth points on $X$.
\begin{enumerate}
\item[i)] We have $\epsilon_F(L; Z)=\sup_{m,e}\frac{p^e-1}{m}$, where the supremum is over
all $m,e\geq 1$ such that $L^m$ separates $p^e$-Frobenius jets at $Z$
\item[ii)] Given any $\delta>0$, there is $e_0\geq 1$ such that
for every $e$ divisible by $e_0$, there is $m$ such that $L^m$ separates
$e$-Frobenius jets at $Z$ and $\frac{p^e-1}{m}>\epsilon_F(L; Z)-\delta$.
\item[iii)] We have $\epsilon_F(L; Z)=\limsup_{m\to\infty}\frac{p^{s_F(L^m;Z)}-1}{m}$.
\end{enumerate}
\end{proposition}

\begin{proof}
The assertion in i) follows from definition, since whenever $L$ separates $p^e$-Frobenius jets
at $Z$, we have $s_F(L^m; Z)\geq e$, hence $\frac{p^e-1}{m}\leq\frac{p^{s_F(L^m; Z)}-1}{m}$.
In order to prove ii), note that by definition we may choose $m_0$ such that
$\frac{p^{e_0}-1}{m_0}>\epsilon_F(L; Z)-\delta$, where $e_0=s_F(L^{m_0}; Z)$. If $e=re_0$, then it follows from
Lemma~\ref{lem1} that if $m=m_0\frac{p^{e_0r}-1}{p^{e_0}-1}$, then
$s_F(L^m; Z)\geq e$. Since $\frac{p^e-1}{m}=\frac{p^{e_0}-1}{m_0}$, the assertion in
ii) follows. We deduce from ii) that there is a sequence $(m_{\ell})_{\ell\geq 1}$
with $m_{\ell}\to \infty$ such that $\lim_{\ell\to\infty}\frac{p^{s_F(L^{m_{\ell}}; Z)}-1}{m_{\ell}}=
\epsilon_F(L; Z)$, which implies iii).
\end{proof}

\begin{remark}
We will see below in Example \ref{projective_space} that we cannot replace the $\lim \sup$ in Proposition \ref{prop1} iii) by a limit.  This stands in contrast with the ordinary Seshadri constant.
\end{remark}

\begin{proposition}\label{prop2}
If $L$ is an ample line bundle on the projective variety $X$, and $Z$ is a reduced set
of smooth points on $X$, then $\epsilon_F(L^r; Z)=r\cdot\epsilon_F(L;Z)$ for every positive integer $r$.
\end{proposition}

\begin{proof}
The inequality ``$\leq$" follows from definition since the left side is the supremum over a smaller set.  Hence we prove the opposite inequality.
Let us fix $j_0$ such that $L^j$ is globally generated for every $j\geq j_0$.
In particular, $L^j$ has sections that do not vanish at any of the points in $Z$, hence
$s_F(L^{m+j}; Z)\geq s_F(L^m;Z)$ for all such $j$ and all $m$.

Given $\delta>0$, let us choose $m$ such that if $e=s_F(L^m; Z)$, then
$\frac{p^e-1}{m}>\epsilon_F(L; Z)-\delta$. For every $i\geq 1$, let
$d_i=\frac{p^{ie}-1}{p^e-1}$. Using Lemma~\ref{lem1}, we deduce that for every $j\geq j_0$
we have
$$s_F(L^{md_i+j}; Z)\geq s_F(L^{md_i}; Z)\geq ie.$$
Let $a_i=\lceil (md_i+j_0)/r\rceil$, where $\lceil u\rceil$ denotes the smallest integer $\geq u$. Note that $ra_i \geq m d_i + j_0$ which implies
$$\frac{p^{s_F(L^{ra_i}; Z)}-1}{a_i}\geq \frac{p^{ie}-1}{a_i}=\frac{p^e-1}{m}
\cdot \frac{d_im}{\lceil (md_i+j_0)/r\rceil}.$$
Applying $\limsup_{i \to \infty}$ to both sides yields
\[
\varepsilon_F(L^r, Z) \geq \limsup_{i \to \infty} \frac{p^{s_F(L^{ra_i}; Z)}-1}{a_i} \geq \frac{p^e-1}{m} \cdot r > (\varepsilon_F(L; Z) - \delta) \cdot r.
\]
Since this holds for every
$\delta>0$, we deduce the inequality ``$\geq$" in the proposition.
\end{proof}

\begin{remark}
Based on the analogy with the usual Seshadri constant, it is natural to expect
that if $L_1$ and $L_2$ are ample line bundles on $X$, and $Z$ is a finite set of smooth points on $X$,
then $\epsilon_F(L_1\otimes L_2; Z)\geq\epsilon_F(L_1; Z)+\epsilon_F(L_2; Z)$.
However, we do not know whether this is, indeed, the case.
\end{remark}

\begin{remark}\label{regularity}
Suppose that $L$ is an ample line bundle on the $n$-dimensional projective variety $X$ and $Z$
is a finite set of smooth points on $X$, defined by the ideal $I_Z$.
Let $m_0$ be such that $L^{m_0}$ is globally generated and $H^i(X,L^m)=0$ for all $i\geq 1$ and $m\geq m_0$. If $m\geq m_0$, then
$s_F(L^m;Z)\geq e$ if and only if $H^1(X,I_Z^{[p^e]}\otimes L^m)=0$. Note that in any case
we have $H^i(X,I_Z^{[p^e]}\otimes L^m)=0$ for $i\geq 2$. It follows that
if $m\geq nm_0$ and
$s_F(L^m; Z)\geq e$, then $I_Z^{[p^e]}\otimes L^{m+m_0}$ is $0$-regular with respect to
$L^{m_0}$ in the sense of Castelnuovo-Mumford regularity. In particular, we see that
$I_Z^{[p^e]}\otimes L^{m+m_0}$ is globally generated (we refer to
\cite[\S 1.8]{positivity} for the basic facts about Castelnuovo-Mumford regularity).
\end{remark}

From now on, we will mostly consider Frobenius-Seshadri constants at a single smooth point
$x\in X$,
in which case we simply write $\epsilon_F(L; x)$.

\begin{example}\label{projective_space}
Consider the case when $X=\PP^n$ and $L=\cO(1)$. Note that for every
$x\in\PP^n$, we have
$s_F(\cO(m); x)\geq e$ if and only if $m\geq n(p^e-1)$. This implies that $\epsilon_F(\cO(m); x)=
\frac{1}{n}$. Note that in this case $s(\cO(m); x)= m$, hence $\epsilon(\cO(1); x)=1$.
This example also shows that the $\lim\sup$ in Proposition~\ref{prop1} iii) might not be limit:
indeed, if we consider $m_e=n(p^e-1)-1$, then
$$\frac{p^{s_F(\cO(m_e); x)}-1}{m_e}=\frac{p^{e-1}-1}{n(p^{e}-1)-1}\to\frac{1}{np}\,\,\text{as}\,\,e\to\infty.$$
\end{example}

\begin{proposition}\label{ineq_usual_Seshadri}
If $L$ is an ample line bundle on the $n$-dimensional projective variety $X$, then for
every smooth point $x\in X$ we have
$$\frac{\epsilon(L;x)}{n}\leq \epsilon_F(L; x)\leq\epsilon(L; x).$$
\end{proposition}

\begin{proof}
Note that if $\frm_x$ is the ideal defining $x$, then we have
\begin{equation}\label{eq1_ineq_usual_Seshadri}
\frm_x^{n(p^e-1)+1}\subseteq \frm_x^{[p^e]}\subseteq\frm_x^{p^e}.
\end{equation}
The second containment in (\ref{eq1_ineq_usual_Seshadri})
implies $s(L^m; x)\geq p^{s_F(L^m; x)}-1$ and we obtain
$\epsilon(L;x)\geq \epsilon_F(L;x)$ using the definition of the Frobenius-Seshadri constant
and (\ref{formula_jet_separation}).

Given $\delta>0$, let $m_0$ be such that
$\frac{s(L^{m_0}; x)}{m_0}>\epsilon(L; x)-\delta$.
Given any positive integer $e$, let $d=\lceil n(p^e-1)/s(L^{m_0}; x)\rceil$.
Therefore we have
$$s(L^{m_0d}; x)\geq d\cdot s(L^{m_0}; x)\geq n(p^e-1).$$
The first containment in (\ref{eq1_ineq_usual_Seshadri}) implies
$s_F(L^{m_0d}; x)\geq e$, and thus
\begin{equation}\label{eq2_ineq_usual_Seshadri}
\epsilon_F(L;x)\geq \frac{p^{s_F(L^{m_0d}; x)}-1}{m_0d}\geq \frac{p^e-1}{m_0\lceil n(p^e-1)/s(L^{m_0};x)\rceil}.
\end{equation}
When $e$ goes to infinity, the right-hand side of (\ref{eq2_ineq_usual_Seshadri})
converges to $\frac{s(L^{m_0};x)}{m_0n}>\frac{1}{n}(\epsilon(L; x)-\delta)$.
We conclude that $\epsilon_F(L;x)\geq\frac{1}{n}(\epsilon(L;x)-\delta)$, and
since this holds for every $\delta>0$, we get $\epsilon_F(L;x)\geq\frac{\epsilon(L;x)}{n}$.
\end{proof}

We note that interpret the Frobenius-Seshadri constant $\epsilon_F(L;x)$ as corresponding to
$\frac{1}{\dim(X)}\epsilon(L;x)$ rather than to $\epsilon(L;x)$. This is justified by
Example~\ref{projective_space}, but also by the results in connection to adjoint linear systems
that we discuss in the next section.

\begin{remark}
If $L$ is ample and globally generated, then the usual Seshadri constant at any
smooth point $x\in X$
satisfies $\epsilon(L; x)\geq 1$ (see \cite[Example~5.1.18]{positivity}). It follows from
Proposition~\ref{ineq_usual_Seshadri} that in this case we have
$\epsilon_F(L; x)\geq\frac{1}{n}$. As Example~\ref{projective_space} shows,
this is optimal.
\end{remark}

\begin{proposition}\label{numerical}
If $L_1$ and $L_2$ are numerically equivalent ample line bundles on the projective variety $X$,
and $x\in X$ is a smooth point, then $\epsilon_F(L_1;x)=\epsilon_F(L_2;x)$.
\end{proposition}

\begin{proof}
Let us write $L_2\simeq L_1\otimes P$, where $P$ is a numerically trivial line bundle.
Note that if $A$ is a very ample line bundle on $X$, then there is $m_0$ such that
$A^m\otimes P^i$ is globally generated for every $m\geq m_0$ and every $i\in\ZZ$.
Indeed, by Fujita's vanishing theorem (see \cite{Fujita}) there is $m_1$ such that
$H^j(X, A^m\otimes L')=0$ for all $j\geq 1$, $m\geq m_1$, and all nef line bundles
$L'$. In particular, if $m\geq m_1+\dim(X)$, we see that $A^m\otimes L'$
is $0$-regular with respect to $A$, in the sense of Castelnuovo-Mumford regularity, hence it is globally generated. Therefore it is enough to take $m_0=m_1+\dim(X)$.

We now apply the above assertion with $A$ being a suitable power of $L_1$ that is very ample,
and deduce that there is a positive integer $j$ such that $L_1^{j}\otimes P^i$
is globally generated for every $i$. In particular, we see that for every positive integer $m$
we have an inclusion
$$H^0(X, L_1^m)\hookrightarrow H^0(X, L_2^{m+j})$$
induced by a section of $L_2^{m+j}\otimes L_1^{-m}\simeq L_1^j\otimes P^{m+j}$
that does not vanish at $x$. This implies $s_F(L_2^{m+j};x)\geq s_F(L_1^m;x)$ for every
$m>0$, and therefore
\begin{equation}\label{eq_numerical}
\frac{p^{s_F(L_2^{m+j};x)}-1}{m+j}\geq\frac{p^{s_F(L_1^m;x)}-1}{m}\cdot\frac{m}{m+j}.
\end{equation}
By Proposition~\ref{prop1} iii), we can find a sequence of $m$'s going to infinity
such that the right-hand side of (\ref{eq_numerical}) converges to $\epsilon_F(L_1;x)$,
and we conclude that $\epsilon_F(L_2;x)\geq\epsilon_F(L_1;x)$. The reverse inequality
follows by symmetry.
\end{proof}

\begin{remark}\label{openness}
Suppose that $L$ is an ample line bundle on the projective variety $X$.
It is standard to see that given any $m$ and $e$, the set of points $x$ in the smooth
locus $X_{\rm sm}$ of $X$ such that
$L^m$ separates $p^e$-Frobenius jets at $x$ is open in $X_{\rm sm}$, hence in $X$. Given any $\alpha>0$, we see
that the set $\{x\in X\mid \epsilon_F(L;x)>\alpha\}$ is open: indeed, by Lemma~\ref{lem1},
this set is the union of the sets of points $x\in X_{\rm sm}$ for which $L^m$ separates $p^e$-Frobenius jets at $x$, the union being over those $e$ and $m$ such that $\frac{p^e-1}{m}>\alpha$.
A formal consequence of this fact is that if the ground field is uncountable, then there is
a maximum among all $\epsilon_F(L;x)$, for $x\in X_{\rm sm}$, and this is achieved on the complement
of a countable union of closed subsets of $X$.
\end{remark}

We will make use of the next proposition in the following section, in order to obtain a criterion for point separation in the case of adjoint line bundles.

\begin{proposition}\label{several_points}
Let $L_1,\ldots,L_r$ be ample line bundles on a projective $n$-dimensional variety $X$, and
$Z=\{x_1,\ldots,x_r\}$ a finite set of smooth points on $X$. If $\epsilon_F(L_i; x_i)\geq \alpha$
for every $i$, then $\epsilon_F(L_1\otimes\ldots\otimes L_r; Z)\geq\alpha$.
\end{proposition}

\begin{proof}
Let $m_0$ be such that for all $m\geq m_0$ the following holds: $L_i^{m}$ is globally generated, and
$H^j(X,L_i^m)=0$ for all $i$ and all $j\geq 1$.
Given $\alpha'<\alpha$, for every $i$ we may find $e_i$ and $m_i$ such that
$L_i^{m_i}$ separates $p^{e_i}$-Frobenius jets, and
$\frac{p^{e_i}-1}{m_i}>\alpha'$.
Note that using Lemma~\ref{lem1} we may always replace $e_i$ by a multiple $te_i$
if we simultaneously replace $m_i$ by $\frac{m_i(p^{te_i}-1)}{p^{e_i}-1}$. Therefore we may assume that
$e_i=e$ for all $i$, and also that
the $m_i \geq nm_0$.

If $m=\max_i m_i$, then $L_i^{m+m_0}$ separates $p^e$-Frobenius jets at $x_i$
(note that $L_i^{m+m_0-m_i}$ is globally generated).
Arguing as in Remark~\ref{regularity}, we may assume that
$\frm_{x_i}^{[p^e]}\otimes L_i^{m+2m_0}$ is $0$-regular with respect to $L_i^{m_0}$,
hence it is globally generated. We can thus find sections
$s_i\in H^0(X,\frm_{x_i}^{[p^e]}\otimes L_i^{m+2m_0})$ that do not vanish at any of the points in
$Z\smallsetminus\{x_i\}$. Since $L_i^{m+2m_0}$ separates $p^e$-Frobenius jets at $x_i$,
the restriction maps
\begin{equation}\label{eq1_several_points}
H^0(X,L_i^{m+2m_0})\to H^0(X,L_i^{m+2m_0}\otimes\cO_X/\frm_{x_i}^{[p^e]})
\end{equation}
are surjective.
Let $L=L_1\otimes\ldots\otimes L_r$ and denote by $I_Z$ the ideal of $Z$.
We deduce that the composition map
\begin{equation}\label{eq2_several_points}
\bigotimes_{i=1}^r H^0(X,L_i^{m+2m_0})\to H^0(X,L^{m+2m_0})\to H^0(X,L^{m+2m_0}\otimes
\cO_X/I_Z^{[p^e]})
\end{equation}
is surjective. Indeed, $H^0(X,L^{m+2m_0}\otimes
\cO_X/I_Z^{[p^e]})=\oplus_{i=1}^rH^0(X,L^{m+2m_0}\otimes\cO_X/\frm_{x_i}^{[p^e]})$,
and the surjectivity of (\ref{eq1_several_points}) implies that the $i^{\rm th}$ component of this direct sum is the image of
$s_1\otimes\ldots\otimes H^0(X,L_i^{m+2m_0})\otimes\ldots\otimes s_r$.
Since the map in (\ref{eq2_several_points}) factors through the restriction map
$$H^0(X,L^{m+2m_0})\to H^0(X,L^{m+2m_0}\otimes \cO_X/I_Z^{[p^e]}),$$ it follows that $L^{m+2m_0}$ separates $p^e$-Frobenius jets at $Z$, hence
$\epsilon_F(L;Z)\geq\frac{p^e-1}{m+2m_0}$. Furthermore, the same inequality will hold if we replace $e$ by $te$ and $m$ by $\frac{m(p^{te}-1)}{p^e-1}$, that is,
\begin{equation}\label{eq3_several_points}
\epsilon_F(L; Z)\geq \frac{1}{\frac{m}{p^e-1}+\frac{2m_0}{p^{te-1}}}.
\end{equation}
When $t$ goes to infinity, the right-hand side of (\ref{eq3_several_points}) converges to
$\frac{p^e-1}{m}$, which is $>\alpha'$. We thus obtain $\epsilon_F(L; Z)>\alpha'$
for every $\alpha'<\alpha$, which gives the statement of the proposition.
\end{proof}

\section{Frobenius-Seshadri constants and  adjoint bundles}

In this section we show how lower bounds on Frobenius-Seshadri constants imply
global generation or very ampleness of adjoint line bundles. As in the previous section,
all our varieties are defined over an algebraically closed field of characteristic $p>0$.

\begin{theorem}\label{adjoint}
Let $L$ be an ample line bundle on a smooth projective variety $X$.
\begin{enumerate}
\item[i)] If $Z=\{x_1,\ldots,x_r\}$ is a finite subset of $X$ such that $\epsilon_F(L; x_i)
>r$ for every $i$, then the restriction map
$$H^0(X,\omega_X\otimes L)\to H^0(X,\omega_X\otimes L\otimes\cO_Z)$$
is surjective.
\item[ii)] If $\epsilon_F(L;x)>1$ for every $x\in X$, then $\omega_X\otimes L$ is globally generated.
\item[iii)] If $\epsilon_F(L;x)>2$ for some $x\in X$, then $\omega_X\otimes L$
defines a rational map that is birational onto its image.
\item[iv)] If $\epsilon_F(L; x)>2$ for every $x\in X$, then $\omega_X\otimes L$
is very ample.
\end{enumerate}
\end{theorem}

In particular, we obtain the following criterion for global generation and very ampleness
of adjoint line bundles in terms of the usual Seshadri constants. Note that the bounds are the same ones as in characteristic zero, when the assertions are proved via vanishing theorems
(see \cite[Proposition~5.1.19]{positivity}).

\begin{corollary}
\label{cor.adjoint}
Let $L$ be an ample line bundle on a smooth $n$-dimensional projective variety $X$.
\begin{enumerate}
\item[i)] If $\epsilon(L; x)>n$, then
$\omega_X\otimes L$ is globally generated at $x$. In particular, if
$\epsilon(L;x)>n$ for every $x\in X$, then $\omega_X\otimes L$ is globally generated.
\item[ii)] If $\epsilon(L; x)>2n$ for some $x\in X$, then $\omega_X\otimes L$
defines a rational map that is birational onto its image.
\item[iii)] If $\epsilon(L; x)>2n$ for every $x\in X$, then
$\omega_X\otimes L$ is very ample.
\end{enumerate}
\end{corollary}

\begin{proof}
For every $x\in X$, we have $\epsilon_F(L;x)\geq\frac{\epsilon(L; x)}{n}$ by
 Proposition~\ref{ineq_usual_Seshadri}. Therefore all assertions follow from
 Theorem~\ref{adjoint}.
\end{proof}

In the proof of Theorem~\ref{adjoint} we will need the following lemma in order to separate tangent vectors.

\begin{lemma}\label{sep_two_jets}
Let $L$ be an ample line bundle on the projective variety $X$, and $x\in X$ a smooth point.
If $\epsilon_F(L;x)>\alpha$ for some $\alpha\in\RR_{>0}$, then we can find
$r$ and $e$ with $\frac{p^e-1}{r}>\frac{\alpha}{2}$ and such that the restriction map
\begin{equation}\label{eq1_sep_two_jets}
\Phi\colon H^0(X,L^{r})\to H^0(X, L^{r}\otimes\cO_X/(\frm_x^2)^{[p^e]})
\end{equation}
is surjective. Furthermore, we may assume that $p^e-1-\frac{\alpha}{2}r$ is arbitrarily large.
\end{lemma}

\begin{proof}
Let $m_0$ be such that for all $m\geq m_0$, the following hold: $L^m$ is globally generated
and $H^i(X, L^m)=0$ for all $i\geq 1$.
It follows from definition that we can find $e$ and $m$ such that
$\frac{p^e-1}{m}>\alpha$ and the restriction map
\begin{equation}\label{eq2_sep_two_jets}
H^0(X, L^m)\to H^0(X,L^m\otimes\cO_X/\frm_x^{[p^e]})
\end{equation}
is surjective. By Lemma~\ref{lem1}, we may replace $e$ by $se$ and
$m$ by $\frac{m(p^{se}-1)}{p^e-1}$ for every $s\geq 1$. In particular, we may assume that
$m$ is arbitrarily large. We may also assume that
$\frac{p^e-1}{2m+m_0}>\frac{\alpha}{2}$. Indeed, we have
$$\lim_{s\to\infty}\frac{p^{se}-1}{2\frac{m(p^{se}-1)}{p^e-1}+m_0}=\frac{p^e-1}{2m}>
\frac{\alpha}{2}.$$
Furthermore, arguing as in Remark~\ref{regularity}, we may assume that
$\frm_x^{[p^e]}\otimes L^{m+m_0}$ is $0$-regular with respect to $L^{m_0}$, hence it is globally generated. In order to prove the first assertion in the lemma, it is enough to show
that  if we take $r=2m+m_0$, then (\ref{eq1_sep_two_jets}) is surjective.

\begin{claim}
It is enough to show that
$W:=H^0(X,L^{2m+m_0}\otimes(\frm_x^{[p^e]}/(\frm_x^2)^{[p^e]}))$ is contained in the image of
$\Phi$.
\end{claim}
\begin{proof}[Proof of claim]
Indeed, it follows from the surjectivity of (\ref{eq2_sep_two_jets}) and the fact that
$L^{m+m_0}$ is globally generated that the restriction map
$$\phi\colon H^0(X, L^{2m+m_0})\to H^0(X,L^{2m+m_0}\otimes\cO_X/\frm_x^{[p^e]})$$
is surjective as well. We deduce that if
$u\in H^0(X,L^{2m+m_0}\otimes\cO_X/(\frm_x^2)^{[p^e]})$
restricts to $\overline{u}\in H^0(X,L^{2m+m_0}\otimes\cO_X/\frm_x^{[p^e]})$, then there
is $s\in H^0(X, L^{2m+m_0})$ such that $\phi(s)=\overline{u}$. Therefore
$u-\Phi(s)\in W$, and if $u-\Phi(s)$ lies in the image of $\Phi$, then so does $u$.
\end{proof}
Let us choose a trivialization of $L$ around $x$, which induces an isomorphism
$L_x\simeq\cO_{X,x}$ that we henceforth tacitly use.
Since $\frm_x^{[p^e]}\otimes L^{m+m_0}$ is globally generated, it follows that given any
$w\in \frm_x^{[p^e]}/(\frm_x^2)^{[p^e]}$, there are $t_1,\ldots,t_d\in H^0(X, \frm_x^{[p^e]}\otimes L^{m+m_0})$ and $f_1,\ldots,f_d\in\cO_{X,x}$ such that
$w=\sum_{i=1}^dt_{i,x}f_{i,x}\,{\rm mod}\,(\frm_x^2)^{[p^e]}$ (here $t_{i,x}$ denotes the restriction of $t_i$ to the stalk at $x$).
Using now the surjectivity of (\ref{eq2_sep_two_jets}), we can find
$\widetilde{f}_i\in H^0(X, L^m)$ such that
$f_{i,x}\equiv \widetilde{f}_{i,x}\,\text{mod}\,\frm_x^{[p^e]}$.
This implies that if $t=\sum_{i=1}^dt_i\otimes\widetilde{f}_i\in H^0(X, L^{2m+m_0})$,
then $w\equiv t_x\,\text{mod}\,(\frm_x^2)^{[p^e]}$. This completes the proof of the fact that
$W$ is contained in the image of $\Phi$.

The last assertion in the lemma
follows from the fact that
$$p^{se}-1-\frac{\alpha}{2}\left(\frac{2m(p^{se}-1)}{p^e-1}+m_0\right)
=(p^{se}-1)\left(1-\alpha\frac{m}{p^e-1}\right)-\frac{m_0\alpha}{2}\to\infty\,\text{when}\,
s\to\infty.$$
\end{proof}

The key ingredient in the proof of Theorem~\ref{adjoint} is the canonical surjective map
$T\colon F_*\omega_X\to\omega_X$. This can be defined as the trace map for
relative duality for the finite flat morphism $F$ (see for example \cite[Chapter 1]{BK}).
It can also be explicitly described as follows. Recall that $X$ is smooth, and suppose that
$U\subseteq X$ is an affine open subset and $y_1,\ldots,y_n\in\cO(U)$ are such that $dy_1,\ldots,dy_n$ give a trivialization of
$\Omega_X$. After possibly replacing $U$ by a suitable open cover, we may assume that
$\cO(U)$ is free over $\cO(U)^p$ with basis
$$\{y_1^{i_1}\cdots y_n^{i_n}\mid 0\leq i_{\ell}\leq p-1\,\text{for all}\,\ell\}.$$
If we put $dy=dy_1\wedge\ldots\wedge dy_n$, then
$T\colon\cO(U)dy\to\cO(U)dy$ is uniquely characterized by the following properties:
\begin{enumerate}
\item[1)] $T(f^p\eta)=f\cdot T(\eta)$ for every $f\in\cO(U)$ and every $\eta\in\cO(U)dy$.
\item[2)] $T(y_1^{i_1}\cdots y_n^{i_n}dy)=y_1^{\frac{i_1-p+1}{p}}\cdots y_n^{\frac{i_n-p+1}{p}}dy$
for every $i_1,\ldots,i_n\in\ZZ_{\geq 0}$ (with the convention that the expression on the right-hand side is zero, unless all exponents are integers).
\end{enumerate}

We will also consider the $e$-iterate of $T$, namely $T^e\colon F^e_*(\omega_X)\to
\omega_X$. Note that if $J$ is an ideal of $\cO_X$, then
$$T^e(F^e_*(J^{[p^e]}\cdot\omega_X))=J\cdot T^e(F^e_*(\omega_X))=J\cdot \omega_X.$$
We can now prove the main result of this section.

\begin{proof}[Proof of Theorem~\ref{adjoint}]
We first prove i). Let $I_Z$ denote the ideal of $Z$.
Note that the hypothesis, together with Proposition~\ref{several_points}
implies that $\epsilon_F(L; Z)=\frac{\epsilon_F(L^r;Z)}{r}>1$.
Therefore we can find $m$ and $e$ such that
$m<p^e-1$ and $L^m$ separates $p^e$-Frobenius jets at $Z$.
Furthermore, by Lemma~\ref{lem1} we may replace $e$ by $se$ and
$m$ by $\frac{m(p^{se}-1)}{p^e-1}$ for every $s\geq 1$. Since
$$p^{se}-1-\frac{m(p^{se}-1)}{p^e-1}=\frac{(p^e-1-m)(p^{se}-1)}{p^e-1}\to\infty\,\,\text{when}\,\,
s\to\infty,$$
it follows that we may assume that $p^e-1-m$ is as large as we want. In particular,
we may assume that $\omega_X\otimes L^{p^e-m}$ is globally generated, in which case
we deduce that $\omega_X\otimes L^{p^e}$ separates $p^e$-Frobenius jets at $Z$.

We now make use of the surjective map $T^e\colon F^e_*(\omega_X)\to \omega_X$.
As we have seen, this induces a surjective map $F^e_*(I_Z^{[p^e]}\omega_X)\to I_Z\omega_X$. Tensoring this by
$L$ and using the projection formula gives a surjective map
$F^e_*(I_Z^{[p^e]} \omega_X\otimes L^{p^e})\to I_Z\omega_X\otimes L$.
Since $F^e_*$ is an exact functor, we obtain
a commutative diagram with exact rows and surjective vertical maps

\[\label{diag2}
\begin{CD}
0@>>>F^e_*(I_Z^{[p^e]} \omega_X\otimes L^{p^e})@>>> F^e_*(\omega_X\otimes L^{p^e})
@>>> F^e_*(\omega_X\otimes L^{p^e}\otimes\cO_X/I_Z^{[p^e]})@>>> 0 \\
@|@VVV@VV{T^e\otimes L}V@VVV@| \\
0@>>> I_Z\omega_X\otimes L@>>>  \omega_X\otimes L@>>>\omega_X\otimes L\otimes
\cO_X/I_Z@>>>0
\end{CD}
\]

By taking global sections we obtain a commutative diagram
\[\label{diag3}
\begin{CD}
H^0(X,\omega_X\otimes L^{p^e})@>{\phi}>>H^0(X,\omega_X\otimes L^{p^e}\otimes\cO_X/I_Z^{[p^e]})\\
@VVV@VV{\rho}V\\
H^0(X,\omega_X\otimes L)@>{\psi}>>H^0(X,\omega_X\otimes L\otimes\cO_X/I_Z)
\end{CD}
\]
in which $\rho$ is surjective. Since $\omega_X\otimes L^{p^e}$ separates $p^e$-Frobenius
jets at $Z$, we have that $\phi$ is surjective, and we thus obtain that $\psi$ is surjective, which completes the proof of i).

The assertion in ii) follows from i), by considering $Z=\{x\}$, for a point $x\in X$.
Under the assumption in iii), it follows from Remark~\ref{openness} that there is an open subset
$U\subseteq X$ such that $\epsilon_F(L; y)>2$ for every $y\in U$. In order to prove iii),
it is enough to show that $\omega_X\otimes L$ separates points and tangent vectors
on $U$. By taking $Z=\{y_1,y_2\}$ for $y_1$ and $y_2$ distinct points in $U$, it follows from
i) that $\omega_X\otimes L$ separates points in $U$. Suppose now that $x\in U$.
The hypothesis together with Lemma~\ref{sep_two_jets} implies
that we can find $e$ and $m$ such that $m<p^e-1$ and such that the restriction map
$$H^0(X,L^m)\to H^0(X,L^m\otimes\cO_X/(\frm_x^2)^{[p^e]})$$
is surjective. Furthermore, since we may assume that $p^e-1-m$ is large enough, we may assume that $\omega_X\otimes L^{p^e-m}$
is globally generated, hence also the restriction map
$$H^0(X,\omega_X\otimes L^{p^e})\xrightarrow{\varphi'} H^0(X,\omega_X\otimes L^{p^e}
\otimes\cO_X/(\frm_x^2)^{[p^e]})$$
is surjective. Arguing as in the proof of i), we obtain a commutative diagram
\[\label{diag4}
\begin{CD}
H^0(X,\omega_X\otimes L^{p^e})@>{\phi'}>>H^0(X,\omega_X\otimes L^{p^e}
\otimes\cO_X/(\frm_x^2)^{[p^e]})\\
@VVV@VV{\rho'}V\\
H^0(X,\omega_X\otimes L)@>{\psi'}>>H^0(X,\omega_X\otimes L\otimes\cO_X/\frm_x^2)
\end{CD}
\]
with $\rho'$ surjective. Since $\phi'$ is surjective, we deduce that
$\psi'$ is surjective, that is, $\omega_X\otimes L$ separates tangent vectors at $x$. This completes the proof of iii). The assertion in iv) now follows from the above argument, by taking
$U=X$.
\end{proof}

\section{Frobenius-Seshadri constants on toric varieties}

It is well-known that the Seshadri constants at the torus-fixed points of a smooth toric variety
can be explicitly described in terms of polyhedral geometry (see \cite{DiRocco} and
\cite{B+}). Our goal in this section is to give
a similar description for the Frobenius-Seshadri constant.
We freely use basic facts and notation on toric varieties from \cite{Fulton}.

Let $N\simeq\ZZ^n$ be a lattice and $M={\rm Hom}_{\ZZ}(N,\ZZ)$ the dual lattice.
We consider a smooth projective toric variety $X$  corresponding to a
 fan $\Delta$ in $N_{\RR}=N\otimes_{\ZZ}\RR$.
 We assume that $X$ is defined over an algebraically closed field $k$
 of characteristic $p>0$. Let $L$ be an
 ample line bundle  on $X$ and $x\in X$
a torus-fixed point. This point corresponds to an $n$-dimensional cone
$\sigma\in\Delta$. We assume that $X$ is smooth and so $\sigma$ is a non-singular cone. We can thus
choose a basis $e_1,\ldots,e_n$ of $N$ such that $\sigma$ is the convex cone generated by
these vectors. If $e_1^*\ldots,e_n^*\in M$ give the dual  basis and we put
$t_i=\chi^{e_i^*}$, then $\cO(U_{\sigma})$
is a polynomial $k$-algebra in $t_1,\ldots,t_n$, and $x\in U_{\sigma}$ is defined by the
ideal $(t_1,\ldots,t_n)$.

There is a torus-invariant divisor $D$ such that $L\simeq\cO_X(D)$. We choose the unique
such $D$ which is effective and whose restriction to the open affine subset $U_{\sigma}$
corresponding to $\sigma$ is zero. Let $P_D\subseteq M_{\RR}\simeq\RR^n$ denote the lattice polytope corresponding
to $D$. Recall that $P_D$ is the convex hull of those $u\in M$ such that
${\rm div}(\chi^u)+D\geq 0$. In particular, we have $P_D\subseteq\RR_+^n$.
Furthermore, since $L$ is ample, it follows that $\Delta$ is the normal fan of $P_D$.
By our normalization of $D$, this implies that the vertex of $P_D$ corresponding to the cone
$\sigma\in\Delta$ is the origin $0\in\RR^n$, and there are precisely $n$ facets of $P_D$
containing $0$, namely
$P_D\cap\{u=(u_1,\ldots,u_n)\in\RR^n\mid u_i=0\}$, for $1\leq i\leq n$.

For every $m\geq 1$, we have a basis of $H^0(X,\cO_X(mD))$
given by $\{\chi^u\mid u\in mP_D\cap M\}$. Furthermore, our choice of $D$
implies that we have a trivialization $\cO_X(D)\vert_{U_{\sigma}}\simeq\cO_{U_{\sigma}}$
such that if $u=(u_1,\ldots,u_n)\in mP\cap M$, then the section $\chi^u\in H^0(X,\cO_X(mD))$ restricts to
$\prod_{i=1}^nt_i^{u_i}\in \cO(U_{\sigma})$.
This proves the following

\begin{lemma}\label{lem_toric}
With the above notation,
$\cO_X(mD)$ separates $p^e$-Frobenius jets at $x$ if and only if
for every $u=(u_1,\ldots,u_n)\in\ZZ_{\geq 0}^n$ with $u_i\leq p^e-1$ for all $i$, we have
$u\in mP_D$.
\end{lemma}

\begin{theorem}\label{toric}
With the above notation, the Frobenius-Seshadri constant of $L\simeq\cO_X(D)$ at $x$ is given by
$$\epsilon_F(L; x)=\max\{r\in\RR_{\geq 0}\mid r\cdot C_n\subseteq P_D\},$$
where $C_n$ is the cube $[0,1]^n\subseteq\RR^n$.
\end{theorem}

\begin{proof}
 Let $M:=
\max\{r\in\RR_{\geq 0}\mid r\cdot C_n\subseteq P\}$.
Lemma~\ref{lem_toric}
gives
$$s_F(L^m; x)=\max\{e\in\ZZ_{\geq 0}\mid (p^e-1)\cdot C_n\subseteq mP_D\}.$$
In particular, for every $m\geq 1$ we have
$$\frac{p^{s_F(L^m; x)}-1}{m}\leq M,$$
hence $\epsilon_F(L; x)\leq M$. Moreover, in order to show that we have equality,
it is enough to show that for every $\delta\in (0,M)$, we can find positive integers $m$ and $e$
such that
\begin{equation}\label{eq_toric}
M-\delta<\frac{p^e-1}{m}\leq M.
\end{equation}
Indeed, in this case $s_F(L^m; x)\geq e$ and therefore
$\epsilon_F(L; x)\geq\frac{p^e-1}{m}>M-\delta$.

It is thus enough to find $e\geq 1$ such that there is an integer $m$ with
$$
\frac{p^e-1}{M}\leq m<\frac{p^e-1}{M-\delta}.
$$
This is clearly possible if
$$\frac{p^e-1}{M-\delta}-\frac{p^e-1}{M}=\frac{\delta(p^e-1)}{M(M-\delta)}>1.$$
This holds for $e\gg 0$, hence we can find $e$ and $m$ such that (\ref{eq_toric})
holds. We thus have $\epsilon_F(L; x)=M$.
\end{proof}

\begin{remark}
It is interesting to compare the formula in Theorem~\ref{toric} with the formula for the
usual Seshadri constant of $L$ at $x$ (see \cite[Corollary~4.2.2]{B+}).
This says that if $Q_n$ denotes the simplex $\{u=
(u_1,\ldots,u_n)\in\RR_{\geq 0}^n\mid u_1+\ldots+u_n\leq 1\}$, then
\begin{equation}\label{eq2_toric}
\epsilon(L; x)=\max\{r\in\RR_{\geq 0}\mid r\cdot Q_n\subseteq P_D\}.
\end{equation}
Of course, one can also rewrite the right-hand side of (\ref{eq2_toric})
as
$$r_0:=\max\{r\in\RR_{\geq 0}\mid re_i^*\in P_D\,\text{for}\,1\leq i\leq n\}.$$
Note that $r_0$ is an integer since $P_D$ is a lattice polytope.

One can prove (\ref{eq2_toric}) arguing as above. Indeed, it follows from the discussion
preceding Lemma~\ref{lem_toric}
that
$L^m$ separates $\ell$-jets at $x$ if and only if for every $u=(u_1,\ldots,u_n)
\in\ZZ_{\geq 0}$ with $\sum_{i=1}^nu_i\leq \ell$, we have $u\in mP_D$.
Therefore
$$s(L^m;x)=\max\{s\in\ZZ_{\geq 0}\mid s\cdot Q_n\subseteq mP_D\}=mr_0$$
for every $m\geq 1$, hence $\epsilon(L; x)=r_0$.
\end{remark}

\begin{example}
Let $M=\ZZ^2$ and consider the convex hull $P$ of the following set of points in $\RR^2$
$$\{(0,0), (1,0), (2,1), (2,2), (1,2), (0,1)\}.$$
The corresponding toric variety $X$ is the blow-up of $\PP^2$ at the three torus-fixed points,
and the line bundle $L$ corresponding to $P$ is the anti-canonical line bundle $\omega_X^{-1}$.
If $x\in X$ is the torus-fixed point corresponding to $0\in P$, then we have
$$\epsilon_F(L; x)=\epsilon(L; x)=1.$$
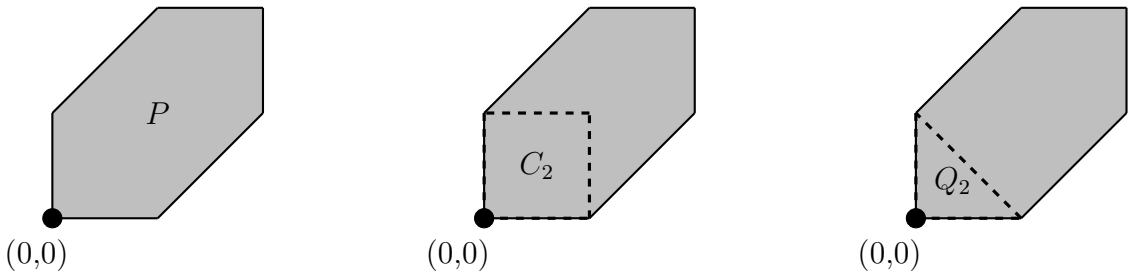
\begin{center}
\begin{tikzpicture}[scale=1.4]
\fill[lightgray] (0,0)--(1,0)--(2,1)--(2,2)--(1,2)--(0,1)--cycle;
\begin{scope}[thick]
\draw(0,0)--(1,0);
\draw(1,0)--(2,1);
\draw(2,1)--(2,2);
\draw(2,2)--(1,2);
\draw(1,2)--(0,1);
\draw(0,1)--(0,0);
\end{scope}
\draw[fill=black] (0,0) circle (0.5ex);
\draw (-0.15, -0.35) node{(0,0)};
\draw(1,1) node{$P$};

\fill[lightgray] (4.1,0)--(5.1,0)--(6.1,1)--(6.1,2)--(5.1,2)--(4.1,1)--cycle;
\begin{scope}[thick]
\draw(4.1,0)--(5.1,0);
\draw(5.1,0)--(6.1,1);
\draw(6.1,1)--(6.1,2);
\draw(6.1,2)--(5.1,2);
\draw(5.1,2)--(4.1,1);
\draw(4.1,1)--(4.1,0);
\end{scope}
\draw[fill=black] (4.1,0) circle (0.5ex);
\draw (3.85, -0.35) node{(0,0)};
\begin{scope}[dashed, very thick]
\draw(4.1,0)--(5.1,0);
\draw(5.1,0)--(5.1,1);
\draw(5.1,1)--(4.1,1);
\draw(4.1,1)--(4.1,0);
\draw(4.6,0.5) node{$C_2$};
\end{scope}

\fill[lightgray] (8.2,0)--(9.2,0)--(10.2,1)--(10.2,2)--(9.2,2)--(8.2,1)--cycle;
\begin{scope}[thick]
\draw(8.2,0)--(9.2,0);
\draw(9.2,0)--(10.2,1);
\draw(10.2,1)--(10.2,2);
\draw(10.2,2)--(9.2,2);
\draw(9.2,2)--(8.2,1);
\draw(8.2,1)--(8.2,0);
\end{scope}
\draw[fill=black] (8.2,0) circle (0.5ex);
\draw (7.95, -0.35) node{(0,0)};

\begin{scope}[dashed, very thick]
\draw(8.2,0)--(9.2,0);
\draw(9.2,0)--(8.2,1);
\draw(8.2,1)--(8.2,0);
\draw(8.55,0.35) node{$Q_2$};
\end{scope}

\end{tikzpicture}
\end{center}
\end{example}

\begin{remark}
The recent interesting preprint \cite{Ito} gives estimates
for the Seshadri constant at a general point on a toric variety (in characteristic zero).
It would be interesting to investigate whether one can obtain similar estimates for the
Frobenius-Seshadri constant when working over a field of positive characteristic.
\end{remark}

\providecommand{\bysame}{\leavevmode \hbox \o3em
{\hrulefill}\thinspace}


\end{document}